\documentclass[sn-mathphys,Numbered]{sn-jnl}

\usepackage{graphicx}
\usepackage{multirow}
\usepackage{amsmath,amssymb,amsfonts}
\usepackage{amsthm}
\usepackage{mathrsfs}
\usepackage[title]{appendix}
\usepackage{xcolor}
\usepackage{textcomp}
\usepackage{manyfoot}
\usepackage{booktabs}
\usepackage{algorithm}
\usepackage{algorithmicx}
\usepackage{algpseudocode}
\usepackage{listings}
\usepackage[latin1]{inputenc}
\usepackage{makecell}
\usepackage{tabularray}
\usepackage{tikz-cd}
\usepackage{mathbbol}
\usepackage{float}
\usepackage{caption}

\SetTblrInner{rowsep=2pt}

\theoremstyle{thmstyleone}
\newtheorem{theorem}{Theorem}[section]
\newtheorem{proposition}[theorem]{Proposition}
\newtheorem{lemma}[theorem]{Lemma}
\newtheorem{corollary}[theorem]{Corollary}

\theoremstyle{thmstyletwo}

\newtheorem{remark}[theorem]{Remark}
\newtheorem{conjecture}[theorem]{Conjecture}

\theoremstyle{thmstylethree}
\newtheorem{definition}[theorem]{Definition}

\everymath{\displaystyle}

\begin{document}

\title[On the Density of Polynomial Mappings Satisfying the Jacobian Conjecture]{On the Density of Polynomial Mappings Satisfying the Jacobian Conjecture}

\author*[1]{\fnm{João Vítor} \sur{Pissolato}}\email{jvpissolato@usp.br}

\affil*[1]{\orgdiv{Department of Mathematics}, \orgname{Instituto de Ciências Matemáticas e de Computação}, \orgaddress{\street{Trabalhador São-Carlense Avenue, 400}, \city{São Carlos}, \postcode{13566-590}, \state{São Paulo}, \country{Brazil}}}

\abstract{The Jacobian Conjecture is a known unsolved problem and it is the problem number 16 of the list ``Mathematical Problems for the Next Century'', made by Stephen Smale, in 1998. The problem asks whether or not the Jacobian matrix of a polynomial mapping $F:\mathbb{C}^n\to\mathbb{C}^n$ at every point being invertible implies that $F$ is an automorphism. The case $n=1$ is trivially true, while the case $n\geq3$ has been recently proven to be false by a counter-example provided by Levent Alpöge, and the case $n=2$ is still an open problem. In this paper, we show that, for all $n\geq1$, there exists a non-empty Zariski dense open set $U$ such that, for all $F\in U$, if the Jacobian matrix of $F$ is invertible, then $F$ is an automorpshim.}


\maketitle	

\section{Introduction}\label{sec1} The Jacobian Conjecture, formulated in full generality for $n$ variables by Ott-Heinrich Keller, in 1939, is a famous open problem. The statement of this conjecture is the following:

\begin{conjecture}
	
	Let $n\geq1$ be a given integer and $F:\mathbb{C}^n\to\mathbb{C}^n$ be a complex polynomial mapping in $n$ variables. If the determinant of its Jacobian matrix is nowhere zero, then $F$ has an inverse that is also a polynomial mapping.
	
\end{conjecture}

The problem for the case $n=2$ has been stated by Ludwig Kraus, in 1884. More details can be seen in \cite{4}.

The case for functions of one variable, that is, for $n=1$, is not difficult to prove to be true. Indeed, if $F:\mathbb{C}\to\mathbb{C}$ is a polynomial mapping $$F(x)=a_nx^n+\cdots+a_2x^2+a_1x+a_0$$ such that $$F'(x)=na_nx^{n-1}+\cdots+2a_2x+a_1$$ is a complex polynomial that has no complex roots, it follows from the Fundamental Theorem of Algebra that $F'(x)$ must be the constant polynomial, that is, $$F'(x)=a_1\neq0 \text{ and }a_i=0, \text{ for }2\leq i\leq n.$$ Therefore, $F(x)=a_1x+a_0$, and such $F$ is an automorphism, with its inverse given by the polynomial $$F^{-1}(x)=\frac{1}{a_1}x-\frac{a_0}{a_1}.$$

For the real case, the problem is known as the Strong Real Jacobian Conjecture, and it has been proven to be false by a counter-example shown by Pinchuk. More details can be seen in \cite{5}.

Over the years, many mathematicians researched about this problem, obtaining many reductions and equivalences, which we refer the works made by Bass, Connell and Wright in \cite{4}.

This conjecture has some connections with the properness property of the mapping. In this paper, we investigate some properties of complex polynomial mappings and prove that a sufficiently generic polynomial mapping is proper. This allows us to conclude that, for such generic polynomials, the Jacobian conjecture is true, which is our main result, stated as follows:

\begin{theorem}
	
	Let $n\geq1$ and $d_1,\ldots,d_n\geq1$ be given integers. Then, there exists a non-empty dense Zariski open set $U\subseteq \Omega_n(d_1,\ldots,d_n)$ such that, for every $F\in U$, if the Jacobian determinant of $F$ is invertible, then $F$ is an automorphism.
	
\end{theorem}

\section{Basic concepts and results about polynomial mappings}\label{sec2}

In this section, we give the basic concepts and fix some notations that we use throughout this work. We follow the notations introduced by Farnik, Jelonek and Ruas in \cite{1}.

\begin{definition}
	
	Given $n\geq 1$ and integers $d_1,\ldots,d_n\geq1$, we define the \textbf{space of polynomial mappings $F:\mathbb{C}^n\to\mathbb{C}^n$ of bounded degrees $(d_1,\ldots,d_n)$}, denoted by $\Omega_n(d_1,\ldots,d_n)$, as the set consisting of all mappings $F:\mathbb{C}^n\to\mathbb{C}^n$ such that $F=(f_1,\ldots,f_n)$, where each $f_i$ is a polynomial in $n$ variables and has degree less than or equal to $d_i$.
	
\end{definition}

\begin{remark}
	
	If we fix an ordering for the monomials in $n$ variables and degrees less than or equal to $d_i$, for $1\leq i\leq n$, we can identify $\Omega_n(d_1,\ldots,d_n)$ as a complex affine space $\mathbb{C}^N$, for some appropriate $N$. This identification is made by considering the coefficients of each $f_i$ to form a $N$-tuple in $\mathbb{C}^N$. Therefore, we can equip the set $\Omega_n(d_1,\ldots,d_n)$ with the Zariski Topology, where a closed set is given by the common zeroes of polynomial expressions of the above coefficients.
	
\end{remark}

\begin{definition}
	
	We denote by $H_n(d_1,\ldots,d_n)$ the subspace of $\Omega_n(d_1,\ldots,d_n)$ consisting of polynomial mappings $F:\mathbb{C}^n\to\mathbb{C}^n$, where $F=(f_1,\ldots,f_n)$ and each $f_i$ is an homogeneous polynomial of degree $d_i$ in $n$ variables.
	
\end{definition}

\begin{remark}
	
	Given $F\in H_n(d_1,\ldots,d_n)$, with $F=(f_1,\ldots,f_n)$, we can write each $f_k$ as \begin{equation}\label{exp1}f_k=\sum_{i_2,i_3,\ldots,i_n}a_{i_2,i_3,\ldots,i_n;k}x_1^{d_k-i_2-\cdots-i_n}x_2^{d_2}x_3^{d_3}\cdots x_n^{d_n},\end{equation} for each $k=1,2,\ldots,n$.
	
\end{remark}

In order to prove our results, we will make use of the known Theorem of the Dimension of the Fibres, which we state below for reference and more details can be found in \cite[4.4 Theorem on the Dimension of the Fibres, p. 228]{2}.

\begin{theorem}[Dimension of the Fibres]\label{tdf}
	
	Let $X$ and $Y$ be two algebraic sets, and let $f:X\to Y$ be a dominant morphism. Then, for all $x\in X$, we have $$\textnormal{dim}(f^{-1}(f(x)))\geq\textnormal{dim}(X)-\textnormal{dim}(Y).$$ Furthermore, there exists a non-empty dense Zariski open set $U\subseteq Y$ such that, for every $y\in U$, the following equality holds: $$\textnormal{dim}(f^{-1}(y))=\textnormal{dim}(X)-\textnormal{dim}(Y).$$
	
\end{theorem}

\section{Main Results}\label{sec3}

We are now in position to state and prove our first result.

\begin{lemma}\label{lemma1}
	
	Let $n\geq1$ and $d_1,\ldots,d_n\geq1$ be given integers. Then, there exists a non-empty dense Zariski open set $U\subseteq H_n(d_1,\ldots,d_n)$ such that, for all $F\in U$, we have $F^{-1}(0)=\{0\}$.
	
\end{lemma}

\begin{proof}
	
	Let us define $$X=\{(p, F)\in\mathbb{C}^n\times H_n(d_1,\ldots,d_n)\mid F(p)=0\},$$ which is an algebraic set, and consider the projections $\pi_1:X\to\mathbb{C}^n$ and $\pi_2:X\to H_n(d_1,\ldots,d_n)$.
	
	Note that all non-zero fibers of $\pi_1$ are isomorphic to $\pi_1^{-1}(e_1)$, where $e_1=(1,0,\ldots,0)$. Indeed, given $p\in\mathbb{C}^n\setminus\{0\}$, extend $\{p_1=p\}$ to a basis $\{p_1,p_2,\ldots,p_n\}$ of the vector space $\mathbb{C}^n$ and consider the linear isomorphism $T:\mathbb{C}^n\to\mathbb{C}^n$ such that $T(p_i)=e_i$, where $\{e_1,e_2,\ldots,e_n\}$ is the canonical basis of $\mathbb{C}^n$. Then, $$\begin{array}{rl}\varphi:&\pi_1^{-1}(e_1)\longrightarrow\pi_1^{-1}(p),\\&(e_1, F)\longmapsto\varphi(e_1, F)=(p, F\circ T)\end{array}$$ is a morphism, since it is given by polynomial expressions of the coefficients of $F$, and its inverse is $$\begin{array}{rl}\psi:&\pi_1^{-1}(p)\longrightarrow\pi_1^{-1}(e_1)\\&(p,F)\longmapsto\psi(p, F)=(e_1, F\circ T^{-1}),\end{array}$$ which is also a morphism for the same reason. Thus, it suffices to consider only the fiber $\pi_1^{-1}(e_1)$.
	
	Using the expression for $F$ given in \ref{exp1}, it follows that the fiber $\pi_1^{-1}(e_1)$ is $$\pi_1^{-1}(e_1)=\{(e_1,F)\in\mathbb{C}^n\times H_n(d_1,\ldots,d_n)\mid a_{0,0,\ldots,0;k}=0,\text{ for }1\leq k\leq n\}.$$ Thus, its dimension is $$\textnormal{dim}(\pi_1^{-1}(e_1))=\textnormal{dim}(H_n(d_1,\ldots,d_n))-n.$$ Now, applying Theorem \ref{tdf} to $\pi_1$, there exists a non-empty dense Zariski open set $V\subseteq\mathbb{C}^n$ such that, for all $p\in V$, we have $$\textnormal{dim}(\pi_1^{-1}(p))=\textnormal{dim}(X)-\textnormal{dim}(\mathbb{C}^n)=\textnormal{dim}(X)-n.$$ Since the affine space $\mathbb{C}^n$ is irreducible, then every two non-empty Zariski open sets have non-empty intersection, so we can take $p\in V\cap(\mathbb{C}^n\setminus\{0\})$ and, from the isomorphism  $\pi_1^{-1}(p)\cong\pi_1^{-1}(e_1)$, we conclude that $$\textnormal{dim}(X)-n=\textnormal{dim}(\pi_1^{-1}(p))=\textnormal{dim}(\pi_1^{-1}(e_1))=\textnormal{dim}(H_n(d_1,\ldots,d_n))-n.$$ In other words, we have that $$\textnormal{dim}(X)-\textnormal{dim}(H_n(d_1,\ldots,d_n))=0.$$ Now, we apply Theorem \ref{tdf} to $\pi_2$ and we obtain a non-empty dense Zariski open set $U\subseteq H_n(d_1,\ldots,d_n)$ such that, for all $F\in U$, we have $$\textnormal{dim}(\pi_2^{-1}(F))=\textnormal{dim}(X)-\textnormal{dim}(H_n(d_1,\ldots,d_n))=0.$$ Finally, let us verify that this set $U$ satisfies our claim. Given $F\in U$, it follows from the homogeneity of $F$ that $\{0\}\subseteq F^{-1}(0)$. Now, if we suppose that there exists a point $p\in F^{-1}(0)\setminus\{0\}$, then $(p,F)\in\pi_2^{-1}(F)$. Once again, it follows from the homogeneity of $F$ that $(\lambda p,F)\in\pi_2^{-1}(F)$, for all $\lambda\in\mathbb{C}$. Thus, the fiber $\pi_2^{-1}(F)$ contains a line, which is a contradiction, since $\textnormal{dim}(\pi_2^{-1}(F))=0$.
	
\end{proof}

\begin{remark}
	
	The above result has a geometric interpretation. For a homogeneous polynomial $f_k$ in $n$ variables, the equation $f_k=0$ defines a hypersurface in $\mathbb{C}^n$ made of lines passing through the origin. The Lemma \ref{lemma1} says that, for sufficiently generic $n$ such hypersurfaces, they only have the origin as a common point.
	
\end{remark}

The condition $F^{-1}(0)=\{0\}$ has an equivalent topological condition. First, let us state and prove another lemma.

\begin{lemma}\label{lemma2} A continuous mapping $F:\mathbb{C}^n\to\mathbb{C}^n$ is proper if, and only if, for every sequence $(x_k)_{k\geq 1}$ of points in $\mathbb{C}^n$ such that $x_k\to\infty$, we have $F(x_k)\to\infty$.
	
\end{lemma}

\begin{proof}
	
	Suppose that $F$ is proper and let $(x_k)_{k\geq1}$ be a sequence of points in $\mathbb{C}^n$ such that $x_k\to\infty$. For the sake of contradiction, let us suppose that $F(x_k)$ does not diverge to $\infty$. Then, there exists $R>0$ such that $||F(x_k)||\leq R$, for every $k$. Since $F$ is proper, we have that the set $F^{-1}({B[0,R]})$ is compact and, in particular, it is bounded, where $B[0,R]$ denotes the closed ball of radius $R$ centered at the origin, which is compact. However, the bounded set $F^{-1}(B[0,R])$ contains the sequence $(x_k)_{k\geq1}$, that is not bounded, and this is a contradiction. Therefore, $F(x_k)\to\infty$.
	
	Suppose that $F$ maps unbounded sequences to unbounded sequences and let us prove that $F$ is proper. Let $K$ be a compact set in $\mathbb{C}^n$. Since $K$ is closed and $F$ is continuous, we already have that $F^{-1}(K)$ is closed. Now, we only need to prove that $F^{-1}(K)$ is bounded. For the sake of contradiction, let us suppose that $F^{-1}(K)$ is not bounded. Then, for each $k\geq1$, we can find $x_k\in F^{-1}(K)$ such that $||x_k||\geq k$. Then, $x_k\to\infty$ and, by hypothesis, $F(x_k)\to\infty$, which is a contradiction, since $(F(x_k))_{k\geq1}$ is a sequence of points in the bounded set $K$.
	
\end{proof}

\begin{proposition}\label{prop1}
	
	Let $F\in H_n(d_1,\ldots,d_n)$. Then, $F:\mathbb{C}^n\to\mathbb{C}^n$ is proper (with respect to the usual topology in $\mathbb{C}^n$) if, and only if, $F^{-1}(0)=\{0\}$.
	
\end{proposition}

\begin{proof}
	
	Suppose that $F$ is proper. Since $F$ is homogeneous, we already have that $\{0\}\subseteq F^{-1}(0)\subseteq\{0\}$. For the sake of contradiction, let us suppose that there exists a point $p\in F^{-1}(0)\setminus\{0\}$. Since $F$ is proper, then the fiber $F^{-1}(0)$ is compact and, in particular, it is bounded. It follows from the homogeneity of $F$ that $\lambda p\in F^{-1}(0)$, for every $\lambda\in\mathbb{C}$. Thus, $F^{-1}(0)$ contains a line, which is contradiction, since $F^{-1}(0)$ is bounded. Therefore, $F^{-1}(0)=\{0\}$.
	
	Conversely, suppose that $F^{-1}(0)=\{0\}$. To prove that $F$ is proper, we will apply Lemma \ref{lemma2}. Let $(x_k)_{k\geq1}$ be a sequence of points in $\mathbb{C}^n$ such that $x_k\to\infty$. Note that $||F(S^{2n-1})||$ is a closed set in $\mathbb{R}$, where $S^{2n-1}$ is the unit sphere in $\mathbb{C}^n$, and $0\notin||F(S^{2n-1})||$, by hypothesis. Since $\mathbb{R}$ is a regular space, there exists $\alpha>0$ such that $$||F(x)||>\alpha>0,\forall x\in S^{2n-1}.$$ Since $x_k\to\infty$, then, by discarding some initial terms, we can suppose that $||x_k||\geq1$, for all $k\geq1$. Then, putting $y_k=\frac{x_k}{||x_k||}\in S^{2n-1}$ and $d=\min\{d_1,\ldots,d_n\}\geq1$, we get\begin{align*}
		||F(x_k)||&=\left|\left|\left(||x_k||^{d_1}f_1(y_k),\ldots,||x_k||^{d_n}f_n(y_k)\right)\right|\right|\\&=||x_k||^{d}\left|\left|\left(||x_k||^{d_1-d}f_1(y_k),\ldots,||x_k||^{d_n-d}f_n(y_k)\right)\right|\right|\\&\geq ||x_k||^d||F(y_k)||\\&\geq||x_k||^d\alpha.
	\end{align*} Since $x_k\geq\infty$, we conclude that $F(x_k)\to\infty$ and that $F$ is proper.
	
\end{proof}

By combining Lemma \ref{lemma1} with Proposition \ref{prop1}, we immediately get the following result:

\begin{corollary}\label{coro1}
	
	Let $n\geq1$ and $d_1,\ldots,d_n\geq1$ be given integers. Then, there exists a non-empty dense Zariski open set $U\subseteq H_n(d_1,\ldots,d_n)$ such that, for every $F\in U$, the mapping $F:\mathbb{C}^n\to\mathbb{C}^n$ is proper with respect to the usual topology in $\mathbb{C}^n$.
	
\end{corollary}

The condition for a complex polynomial mapping to be proper is very strong, since for such a mapping, the set $F^{-1}(p)$ must be compact, but it is also an algebraic set. Below, we show that the only algebraic sets that are compact in the usual topology are the ones that consist of finitely many point.

\begin{proposition}\label{prop2}
	
	An algebraic set $X\subseteq\mathbb{C}^n$ is compact if, and only if, $X$ is a finite set $\{p_1,\ldots,p_m\}$.
	
\end{proposition}

\begin{proof}
	
	Suppose that $X$ is compact and let $X=\cup_{i=1}^mY$ be its decomposition in irreducible components. Since $\mathbb{C}$ is an algebraically closed field, we can apply the Geometric Form of Noether's Normalization Lemma in \cite[Geometric form of Noether's normalization lemma, p. 42]{3} for each $i=1,\ldots,m$ to find a surjective morphism $\pi_i:Y_i\to\mathbb{C}^{\textnormal{dim}(Y_i)}$. Since $X$ is compact and $Y_i$ is closed, then $Y_i$ is also compact, and it follows from the continuity of $\pi_i$ with respect to the usual topology, since it is given by polynomial expressions, that $\pi_i(Y_i)=\mathbb{C}^{\textnormal{dim}(Y_i)}$ is compact, and this is only possible if $\textnormal{dim}(Y_i)=0$. We conclude that $Y_i=\{p_i\}$ is a single point and $X=\{p_1,\ldots,p_m\}$ is a finite set.
	
	For the converse, given a covering $X\subseteq\cup_{i\in I}U_i$ of open sets, we have that, for each $j=1,\ldots,m$, there exists $i_j\in I$ such that $p_j\in U_{i_j}$. Thus, $X\subseteq\cup_{j=1}^mU_{i_j}$. Therefore, $X$ is compact.
	
\end{proof}

\begin{remark}
	
	From Lemma \ref{lemma1}, we have $F^{-1}(0)=\{0\}$, for generic $F\in H_n(d_1,\ldots,d_n)$. By Proposition \ref{prop1}, the mapping $F$ is proper. So, every fiber $F^{-1}(p)$ is a compact set in $\mathbb{C}^n$. It follows from Proposition \ref{prop2} that every fiber $F^{-1}(p)$ is a finite set.
	
	It is important to note that this implication does not work in the real case. For example, consider $F:\mathbb{R}^2\to\mathbb{R}^2$ given by $f(x,y)=(x^2+y^2, x^2+y^2)$. It is not difficult to check that $F^{-1}(0)=\{0\}$ and that $F$ is a proper mapping. However, given a real number $r>0$, the fiber $F^{-1}(r^2,r^2)$ is a circle of radius $r$ centered at the origin, which is not a finite set. 
	
\end{remark}

We now extend the result stated in Corollary \ref{coro1} to the space of polynomial mappings with bounded degrees $\Omega_n(d_1,\ldots,d_n)$.

\begin{lemma}\label{lemma3}
	
	Let $n\geq1$ and $d_1,\ldots,d_n\geq1$ be given integers. Then, there exists a non-empty dense Zariski open set $U\subseteq\Omega_n(d_1,\ldots,d_n)$ such that, for every $F\in U$, the mapping $F:\mathbb{C}^n\to\mathbb{C}^n$ is proper with respect to the usual topology in $\mathbb{C}^n$.
	
\end{lemma}

\begin{proof}
	
	Let $\pi:\Omega_n(d_1,\ldots,d_n)\to H_n(d_1,\ldots,d_n)$ be the projection of the homogeneous part of each $f_i$ with degree $d_i$, which is a morphism. By Corollary \ref{coro1}, there exists a non-empty dense Zariski open set $V\subseteq H_n(d_1,\ldots,d_n)$ such that, for every $F\in V$, the mapping $F:\mathbb{C}^n\to\mathbb{C}^n$ is proper with respect to the usual topology.
	
	Consider $U=\pi^{-1}(V)$, which is a non-empty Zariski open set in $\Omega_n(d_1,\ldots,d_n)$ because $\pi$ is continuous. Since $\pi$ is also open and surjective, it follows that $$\overline{U}=\overline{\pi^{-1}(V)}=\pi^{-1}\left(\overline{V}\right)=\pi^{-1}(H_n(d_1,\ldots,d_n))=\Omega_n(d_1,\ldots,d_n).$$ Thus, $U$ is dense in $\Omega_n(d_1,\ldots,d_n)$. Furthermore, given $F\in U$ and a sequence $(x_k)_{k\geq1}$ in $\mathbb{C}^n$, with $x_k\to\infty$ in $\mathbb{C}^n$, we have that $F(x_k)\to\infty$, because the homogeneous part of degrees $(d_1,\ldots,d_n)$ of $F$ satisfies this condition, by Lemma \ref{lemma2}, since $\pi(F)\in V$ is proper. It follows again from Lemma \ref{lemma2} that $F$ is proper. Therefore, $U$ is the desired set.
	
\end{proof}

In \cite[Theorem 2.1, p. 294]{4}, the authors prove that, if the Jacobian determinant of $F$ is invertible, then $F$ being invertible is equivalent to $F$ being a proper map. Combining this equivalence with Lemma \ref{lemma3}, we conclude our main result:

\begin{theorem}\label{teo1}
	
	Let $n\geq1$ and $d_1,\ldots,d_n\geq1$ be given integers. Then, there exists a non-empty dense Zariski open set $U\subseteq \Omega_n(d_1,\ldots,d_n)$ such that, for every $F\in U$, if the Jacobian determinant of $F$ is invertible, then $F$ is an automorphism.
	
\end{theorem}

In other words, our main result states that, whether or not the Jacobian conjecture is true, it holds for sufficiently generic polynomial mappings. The problem is determining if that non-empty dense Zariski open set can be expanded further to be the entire space of polynomial mappings.

\section{The Levent Alpöge's Example}

Recently, Levent Alpöge published a counter-example for the Jacobian conjecture in the case $n=3$. The counter-example is $F=(f_1,f_2,f_3):\mathbb{C}^3\to\mathbb{C}^3$, where \begin{align*}
	f_1&=(1+xy)^3z + y^2(1+xy)(4+3xy),\\
	f_2&=y + 3x(1+xy)^2z + 3xy^2(4+3xy),\\
	f_3&=2x - 3x^2y - x^3z.
\end{align*}

By direct computation, it is not difficult to see that the Jacobian determinant of $F$ is constant equal to $-2$. However, $F$ is not injective, since $$\left(0,0,-\frac{1}{4}\right), \left(1,-\frac{3}{2},\frac{13}{2}\right), \left(-1,\frac{3}{2}, \frac{13}{2}\right)\in F^{-1}\left(-\frac{1}{4},0,0\right).$$ Therefore, for $n\geq3$, it is possible to construct a counter-example by considering the lifting $G:\mathbb{C}^n\to\mathbb{C}^n$ given by $$G(x,y,z,w_4,\ldots,w_n)=(f_1(x,y,z), f_2(x,y,z), f_3(x,y,z), w_4,\ldots,w_n).$$ However, the conjecture is still an open problem for $n=2$.

Let us analyze this example with a geometric point of view. Take $d_1=\textnormal{deg}(f_1)=7$, $d_2=\textnormal{deg}(f_2)=6$, $d_3=\textnormal{deg}(f_3)=5$ and consider the homogeneous part of $F$ of degree $(d_1,d_2,d_3)$ $$\pi(F)=(x^3y^3z, 3x^3y^2z, -x^3z).$$ By taking the primary decomposition of the ideal $I$ generated by the coordinate functions of $\pi(F)$, we have that $$\langle  x^3y^3z, 3x^3y^2z, -x^3z\rangle=\langle x^3\rangle\cap\langle z\rangle.$$ Therefore, the fiber $\pi(F)^{-1}(0)$ is an algebraic set of dimension $2$ and $\pi(F)$ is not a proper mapping. It shows that $\pi(F)$ is not in the non-empty dense Zariski open set $U$ obtained in Corollary \ref{coro1}.

\section{Conclusion}

In this paper, we show that, even when the Jacobian conjecture is not true in general, it holds for sufficiently generic polynomials. In the case $n=2$, where the question is still an open problem, we can say that, if there exists a counter-example, it must be rare, that is, the coefficients of this polynomial mapping must not all vary in a Zariski open set.

\backmatter

\bmhead{Acknowledgments}

The author wishes to extend his heartfelt appreciation to FAPESP (S\~{a}o Paulo Research Foundation) for their invaluable partial support, process numbers 2020/14442-9 and 2023/03086-5, which played an important role in the development of this scientific research.

\bibliography{sn-bibliography}

\end{document}